\documentclass[12pt]{amsart}

\usepackage{amsthm}

\def\pmatrix{\left(\begin{matrix}}
\def\endpmatrix{\end{matrix}\right)}

\def\P{{\mathbb P}}

\def\Z{{\mathbb Z}}
\def\C{{\mathbb C}}

\def\p{\partial}

\def\e{\varepsilon}

\def\e{\epsilon}

\def\A{{\mathcal A}}

\theoremstyle{plain}
\newtheorem{prop}{Proposition}
\newtheorem{thm}{Theorem}

\theoremstyle{definition}
\newtheorem*{rem}{Remark}

\title{A special case of the $\Gamma_{00}$ conjecture}
\author{Samuel Grushevsky}
\thanks{Research is supported in part by National Science Foundation under the grant DMS-05-55867}
\address{Mathematics Department,\\ Princeton University,\\ Fine Hall, Washington Road,\\ Princeton, NJ 08544, USA}

\begin{document}
\begin{abstract}
In this paper we prove the $\Gamma_{00}$ conjecture of van Geemen and van der Geer \cite{vgvdg} under the additional assumption that the matrix of coefficients of the tangent has rank at most 2 (see theorem \ref{main} for a precise formulation). This assumption is satisfied by Jacobians (see proposition \ref{jacOk}), and thus our result gives a characterization of the locus of Jacobians among all principally polarized abelian varieties.

The proof is by reduction to the (stronger version of the) characterization of Jacobians by semidegenerate trisecants, i.e. by the existence of lines tangent to the Kummer variety at one point and intersecting it in another, proven by Krichever in \cite{krtri} in his proof of Welters' \cite{W} trisecant conjecture.
\end{abstract}

\maketitle

The Schottky problem is the question of characterizing Jacobians of algebraic curves among all ppavs --- complex principally polarized abelian varieties $(A,\Theta)$. Many approaches and solutions (or weak solutions --- those that only characterize the Jacobian locus up to other irreducible components) to the problem have been developed, via the singular locus of the theta divisor (Andreotti-Mayer \cite{AM}), representability of the curves of the minimal class (Matsusaka \cite{M} and Ran \cite{R}), the Kadomtsev-Petviashvili equation (Shiota \cite{S}), etc.

One approach that led to various characterizations of Jacobians is via the geometry of the linear system $|2\Theta|$. The linear subsystem $\Gamma_{00}\subset |2\Theta|$ is defined to consist of those sections that vanish with multiplicity at least 4 at the origin.

It was essentially known to Frobenius, and more recently proven by van Geemen and van der Geer \cite{vgvdg}, Dubrovin \cite{du}, Fay \cite{F2}, Gunning \cite{G2} that for Jacobians of Riemann surfaces all elements of $\Gamma_{00}$ vanish along $C-C$. This led van Geemen and van der Geer to make the following

\smallskip\noindent
{\bf $\bf\Gamma_{00}$ Conjecture.} (\cite{vgvdg})
{\it If for an indecomposable (not a product of lower-dimensional) ppav $(A,\Theta)$ the base locus $Bs(\Gamma_{00})\ne\lbrace 0\rbrace$, then $(A,\Theta)$ is a Jacobian.}

Conjecture 2 in \cite{vgvdg} is that $A$ is a Jacobian if $\dim Bs(\Gamma_{00})\ge 2$, but the stronger version stated above was further conjectured --- see the discussion by Donagi in \cite{do}. In \cite{vgvdg}, conjecture 1, van Geemen and van der Geer further conjectured that for a Jacobian of a curve $C$ the base locus $Bs(\Gamma_{00})$ is in fact equal to $C-C$ (only containment was known classically). This equality on the level of sets was proven for $g>4$ by Welters in \cite{W2} (he also showed that for $g=4$ the base locus of $\Gamma_{00}$ contains two more points). Izadi in \cite{I1} showed that $Bs(\Gamma_{00})=C-C$ holds scheme-theoretically.

Beyond providing a solution to the Schottky problem, the validity of the $\Gamma_{00}$ conjecture would have implications for example for characterization of hyperelliptic Jacobians by their Seshadri constants (see Debarre \cite{de}). The $\Gamma_{00}$ conjecture was proven for a generic ppav for $g=5$ or $g\ge 14$ by Beauville, Debarre, Donagi, and van der Geer in \cite{4p}. Izadi proved in \cite{I2} the conjecture for all 4-dimensional ppavs, and obtained further results on $Bs(\Gamma_{00})$ for Prym varieties in \cite{I3}.

\smallskip
In fact a slightly more precise result is known for Jacobians: from the explicit formulas for Jacobians in \cite{F2} and \cite{G2} more information can be obtained than just the statement that $C-C\subset Bs(\Gamma_{00})$ (see proposition \ref{jacOk} for the precise formulation).
In this paper we prove that this refined version of the $\Gamma_{00}$ conjecture (a certain extra condition on the linear dependence, see theorem \ref{main} for a precise formulation) characterizes Jacobians. The proof is by reducing this conjecture to a characterization of Jacobians by the validity of a certain difference-differential equation for the theta function on the theta divisor, proven by Krichever in \cite{krtri} along the way of his proof of the semidegenerate case of the trisecant conjecture.

\section{$\Gamma_{00}$ conjecture as a condition on the Kummer variety}
We now review the terminology and notations regarding the $|2\Theta|$ linear system and the Kummer variety, and formulate the precise version of the $\Gamma_{00}$ conjecture that we prove. There are no new results in this section; we gather the notations, and give proofs for completeness.

Given a ppav $(A,\Theta)$ the Kummer map is the embedding
$$
  |2\Theta|:A/\pm 1\hookrightarrow \P^{2^g-1}
$$
and the Kummer variety is the image of this map. We denote the Kummer variety $K(A)$ and the Kummer embedding map by $K$. It is often convenient to choose a specific basis of the space of sections $H^0(A,2\Theta)$. If the abelian variety is given as $A=\C^g/\Z^g+B\Z^g$ for some symmetric $g\times g$ matrix $B$ with positive-definite imaginary part, and the polarization $\Theta$ is the divisor of the theta function
$$
 \theta(B,z):=\sum\limits_{n\in \Z^g} e^{2\pi i(z,m)+\pi i(B n,n)},
$$
where $(\cdot,\cdot)$ denotes the complex scalar product (not Hermitian!) of vectors in $\C^g$,  then the basis of $H^0(A,2\Theta)$ consists of
$$
  \Theta[\e](B,z):= \sum\limits_{n\in\Z^g} e^{2\pi i\left(2n+\e,z\right)+\pi i \left(2n+\e,B(n+\frac{\e}{2})\right)}
$$
for all $\e\in(\Z/2\Z)^g$. We will suppress $B$ in the notations, as it will not vary in our discussion.

\smallskip
Fay \cite{F} and Gunning \cite{G,G2} observed that the Kummer image of a Jacobian has many trisecant lines. Gunning \cite{G} proved that the existence of a certain one-dimensional family of trisecants characterizes Jacobians; Welters \cite{W} proved that a jet of such a curve suffices, and conjectured that the existence of a single trisecant already characterizes Jacobians. In view of this conjecture one can also ask whether degenerate trisecants characterize Jacobians: here the most degenerate trisecant corresponds to the case of a flex line, and the semidegenerate case is that of a line tangent to the Kummer image and intersecting it in some other point. The flex line case of the trisecant conjecture was proven by Krichever in \cite{krflex}, and the remaining two cases --- in \cite{krtri}.

\smallskip
Notice that all sections of $|2\Theta|$ are even functions on the universal cover $\C^g$ of the abelian variety. Thus a section of $|2\Theta|$ vanishes at zero to order 4 if and only if it vanishes together with all its second derivatives, as the first and third order derivatives are automatically zero. We thus have
$$
   \Gamma_{00}=\lbrace f\in |2\Theta|\ :\ 0=f(0)=\p_{z_i}\p_{z_j}f (0)=0\ \forall i,j\rbrace,
$$
where we denote differentiation by $\p_z=\frac{\p}{\p z}$.
Since any section $f\in|2\Theta|$ is a linear combination $f=\sum f_\e\Theta[\e]$, the conditions for $f$ to lie in $\Gamma_{00}$ are linear conditions on the coefficients $f_\e$, and we get
$$
  Bs(\Gamma_{00})=\left\lbrace z\in A\ :\  K(z)\in\left\langle K(0),\p_{z_i}\p_{z_j}K(0)\right\rangle_{\rm linear\ span}\right\rbrace.
$$
Thus we get a reformulation (given in \cite{vgvdg})

\smallskip\noindent
{\bf $\bf\Gamma_{00}$ Conjecture} (equivalent reformulation). {\it If for an indecomposable ppav $(A,\Theta)$ there exists a point $P\in A\setminus\lbrace0\rbrace$, and some numbers $c, c_{ij}\in\C$, for $1\le i,j\le g$ such that
$$
 K(P)=c K(0)+\sum\limits_{i,j=1}^g c_{ij}\p_{z_i}\p_{z_j} K(0),
$$
then $(A,\Theta)$ is a Jacobian.}

Notice that since for taking partial derivatives the order of operations does not matter, we can assume that the matrix $c_{ij}$ is symmetric.
Moreover, is is known that for Jacobians the matrix $c_{ij}$ is in fact of rank 1 (or if we want it to be symmetric, of rank 2). In our notations the result is the following

\begin{prop}[{\rm Dubrovin \cite{du}, Fay \cite{F2} theorem 2.5, van Geemen and van der Geer \cite{vgvdg} proposition 2.1, Gunning \cite{G2} corollary 1, essentially known to Frobenius}]\label{jacOk}
Choose any points $p,q$ on the Abel-Jacobi image of a curve $C$ in its Jacobian $Jac(C)$, and let $U,V$ be the tangent vectors to the image of $C$ at $p$ and $q$ respectively. Then there exist constants $b,c\in\C$ such that
$$
 K(p-q)=c K(0)+b\p_U\p_V K(0).
$$
\end{prop}
\begin{proof}
We sketch the proof of this proposition to emphasize the relation (but a priori not an equivalence) of this statement and the existence of trisecants. To prove this, one uses a suitable degeneration of Fay's trisecant formula \cite{F,G}. Indeed, it is known that for arbitrary points $p,p_1,p_2,p_3$ on the Abel-Jacobi image of a curve the 3 points
$$
 K\left(\frac{p+p_1-p_2-p_3}{2}\right),
 K\left(\frac{p+p_2-p_3-p_1}{2}\right),
 K\left(\frac{p+p_3-p_1-p_2}{2}\right)
$$
are collinear, i.e. linearly dependent. The ``semidegenerate'' case of this conjecture is when such a trisecant line degenerates to a tangent line to the Abel-Jacobi image at some point, intersecting it also in some other point, i.e. is the case when, say, $p_2,p_3\to q$ for $p,p_1$ fixed. In this case the limiting statement is that
$$
 K\left(\frac{p+p_1-2q}{2}\right), K(p-p_1), \p_U K(p-p_1)
$$
are linearly dependent, where $U$ is the tangent vector at $q$ to $C\subset Jac(C)$. If we further degenerate this condition as $p\to p_1$, and again take the first term in the Taylor expansion of $\p_U K(p-p_1)$ (which, notice, vanishes at 0 by parity), in the limit we get the statement that
$$
  K(p-q), K(0), \p_V\p_U K(0),
$$
are collinear, where $V$ is the tangent vector at $p$ to $C\subset Jac(C)$, which is exactly the statement of the proposition.
\end{proof}

\begin{rem}
Note that while we can degenerate a family of trisecants, there is no a priori reason why the degenerate condition would imply the less degenerate one, and thus it is not at all clear that the statement of the $\Gamma_{00}$ conjecture implies the existence of (semidegenerate) trisecants. Similarly, there is no a priori way to show that the existence of just one less degenerate trisecant implies the existence of any more degenerate ones.
\end{rem}

The purpose of this article is to prove that this version of the $\Gamma_{00}$ conjecture --- if we additionally require the symmetric matrix $\lbrace c_{ij}\rbrace$ to have rank at most two --- characterizes Jacobians. The fact that a similar condition (with $U=V$) is related to the KP equation is discussed by van Geemen and van der Geer \cite{vgvdg}, observation 4.10.

\begin{thm}[Main theorem]\label{main}
This special case (the symmetric matrix $\lbrace c_{ij}\rbrace$ having rank at most two, so that we can write $c=(U\otimes V+V\otimes U)/2$) of the $\Gamma_{00}$ conjecture characterizes Jacobians, i.e. an indecomposable ppav $(A,\Theta)$ is a Jacobian if and only if there exist a point $P\in A\setminus\lbrace 0\rbrace,$ vectors  $U,V\in\C^g$, and a constant $c\in\C$ such that
\begin{equation}\label{gamma00}
  K(P)=c K(0)+\p_U\p_V K(0).
\end{equation}
\end{thm}

\section{$\Gamma_{00}$ conjecture as a difference-differential equation on the theta divisor}
Similar to the ideas of \cite{krflex},\cite{krtri}, and \cite{pryms} we first use Riemann's bilinear addition theorem
\begin{equation}\label{bilinear}
 \theta(z+Z)\theta(z-Z)=\sum\limits_{\e\in(\Z/2\Z)^g}\Theta[\e](z)\Theta[\e](Z) =:K(z)\cdot K(Z)
\end{equation}
(where the dot denotes the scalar product in $\C^{2^g}$, and we consider the lifting of the map $K$ to the universal cover $\C^g$ of $A$) to rewrite equation (\ref{gamma00}) as a functional equation for the theta function. Indeed, for arbitrary $z\in\C^g$ let us take the scalar product of (\ref{gamma00}) with $K(z)$ and apply (\ref{bilinear}) on both sides to get
$$
  \theta(z+P)\theta(z-P)=c\theta^2(z)+K(z)\cdot\p_U\p_V K(0).
$$
To express the last term in the equation in terms of the theta function, we take the $\p_U\p_V$ derivative of (\ref{bilinear}) in the $Z$ variable, obtaining
$$
 K(z)\cdot\p_U\p_V K(Z)=\p_U\p_V\theta(z+Z)\,\theta(z-Z)+ \theta(z+Z)\,\p_V\p_V\theta(z-Z)
$$
$$
 -\p_U\theta(z+Z)\p_V\theta(z-Z)-\p_V\theta (z+Z)\p_U\theta(z-Z).
$$
Setting now $Z=0$ we get
$$
  K(z)\cdot\p_U\p_V K(0)=2 \theta(z)\p_U\p_V\theta(z)-2\p_U\theta (z)\p_V\theta(z).
$$
We now substitute this expression back in to get
$$
  \theta(z+P)\theta(z-P)=c\theta^2(z)+2\theta(z)\p_U\p_V\theta(z)-2 \p_U\theta(z)\p_V\theta(z).
$$
We will now replace the argument $z$ by $Ux+Vy+Pt+Z$, for $x,y,t\in\C$, and denote
\begin{equation}\label{taudef}
 \tau(x,y,t):=\theta(Ux+Vy+Pt+Z);
\end{equation}
\begin{equation}\label{udef}
 u(x,y,t):=2\p_x\p_y\ln\tau(x,y,t),
\end{equation}
so that the above equation becomes
$$
 \tau(x,y,t+1)\tau(x,y,t-1)=\tau^2(x,y,t)\left(c+u(x,y,t)\right).
$$
Dividing through by $\tau(x,y,t)\tau(x,y,t-1)$ and denoting
\begin{equation}\label{psidef}
 \psi(x,y,t):=\frac{\tau(x,y,t)}{\tau(x,y,t-1)}
\end{equation}
we get the following
\begin{prop}
The assumption of the main theorem, i.e. equation {\rm (\ref{gamma00})} having a solution, is equivalent to the equation
\begin{equation}\label{upsi}
  (c+u(x,y,t)-T)\psi(x,y,t)=0
\end{equation}
being satisfied with $\tau,u,$ and $\psi$ given by (\ref{taudef}),(\ref{udef}),(\ref{psidef}), with $T=e^{\p_t}$ being the operator of shifting $t$ by $1$.
\end{prop}
Similar to the previous work on the subject, we now investigate the condition for equation (\ref{upsi}) to have any solution with $\psi(x,y,t)$ having a simple pole along the divisor of $\tau(x,y,t-1)$ --- without requiring $\psi$ to have the exact form of (\ref{psidef}).
\begin{prop}
If for an indecomposable ppav equation {\rm (\ref{gamma00})} is satisfied then for any $z$ on the theta divisor
\begin{equation}\label{thetaidentity}
\theta_{UU}(z)\theta(z-P)\theta(z+P)=\theta_U(z)\theta_U(z-P)\theta(z+P) +\theta_U(z)\theta_U(z+P)\theta(z-P),
\end{equation}
where the subscript $U$ denotes the directional derivative in direction $U$ with respect to $z$.
\end{prop}
\begin{proof}
To show this, suppose that the function $\tau(x,y,t)$ locally has a root at $x=\eta(y,t)$, where $\eta(y,t)$ is some function, and suppose that locally $\tau(\eta(y,t),y,t-1)\tau(\eta(y,t),y,t+1)\ne 0$ (this is true generically on the theta divisor $\Theta$, since for indecomposable ppavs we would otherwise have $\Theta=\Theta+P$ or $\Theta=\Theta-P$, but the corresponding line bundles are not linearly equivalent). Then locally near the point $x=\eta(y,t)$ from the Taylor expansion
$$
  \tau(x,y,t)=\alpha(x-\eta)+O((x-\eta)^2)
$$
where $\alpha\in\C$ we get from (\ref{udef}) the Taylor series
$$
 u(x,y,t)=2\p_x\p_y\ln\tau(x,y,t)=2\p_x\p_y\ln(\alpha(x-\eta)+O((x-\eta)^2)
$$
$$
 =2\p_y\left(\frac{1}{x-\eta}+O(1)\right)=  \frac{2\eta_y}{(x-\eta)^2}+O(1),
$$
while from (\ref{psidef}) we obtain
$$
 \psi(x,y,t)=\frac{\tau(x,y,t)}{\tau(x,y,t-1)}=a(x-\eta)+b(x-\eta)^2+O((x-\eta)^3),
$$
$$
 T\psi(x,y,t)=\frac{\tau(x,y,t+1)}{\tau(x,y,t)}=\frac{A}{x-\eta}+B+O(x-\eta),
$$
where $a,b,A,B$ are in fact some values of derivatives that we will eventually need to compute. When we substitute all of these expansions into the left-hand-side of (\ref{upsi}) and expand, we get
$$
 (c+u-T)\psi=\frac{2\eta_y}{(x-\eta)^2}\,a(x-\eta)-\frac{A}{x-\eta}+2b\eta_y-B+O(x-\eta).
$$
By equation (\ref{upsi}) this is identically zero, and so we must have
$$
 2a\eta_y=A;\qquad 2b\eta_y=B.
$$
We eliminate $\eta_y$ from these expressions to get the relation $Ab=aB$. Now compute the coefficients $a,b,A,B$ in the Taylor expansions. It is notationally convenient to go back from $\tau$ to $\theta$ at this point, noting that $\tau_x=\theta_U$, and to denote $z:=Ux+Vy+Pt+Z$ the point lying on the theta divisor (for $x=\eta(y,t)$). We then have by differentiating (\ref{psidef})
$$
 a=\p_x\psi(x,y,t)|_{x=\eta(y,t)}=\p_x\frac{\tau(x,y,t)}{\tau(x,y,t-1)}|_{x=\eta(y,t)}
$$
$$ =\p_U\frac{\theta(z)}{\theta(z-P)}|_{\lbrace \theta(z)=0\rbrace}=
 \frac{\theta_U(z)}{\theta(z-P)},
$$
$$
  b=\p_x\p_x\psi(x,y,t)|_{x=\eta(y,t)}=\p_U\left( \frac{\theta_U(z))}{\theta(z-P)}-\frac{\theta_U(z-P)\theta(z)} {\theta^2(z-P)}\right)|_{\lbrace \theta(z)=0\rbrace}
$$
$$
  = \frac{\theta_{UU}(z)}{\theta(z-P)}- 2\frac{\theta_U(z-P)\theta_U(z)}{\theta^2(z-P)}
$$
(notice that one term vanished in each formula since $\tau(\eta,y,t)=0$). To compute $A$ and $B$, it is convenient to think of the expansion of $$
 (T\psi(x,y,t))^{-1}=\frac{\tau(\eta,y,t)}{\tau(\eta,y,t+1)},
$$
which is the same as above with $z-P$ replaced by $z+P$, while on the other hand it is given by
$$
 (T\psi(x,y,t))^{-1}=\frac{1}{A}(x-\eta)-\frac{B}{A^2}(x-\eta)+O((x-\eta)^3).
$$
This yields
$$
 \frac{1}{A}=\frac{\theta_U(z)}{\theta(z+P)};\quad -\frac{B}{A^2}=\frac{\theta_{UU}(z)}{\theta(z+P)}- 2\frac{\theta_U(z+P)\theta_U(z)}{\theta^2(z+P)}
$$
Substituting all of these back into $Ab=aB$, which it is convenient to write as $\frac{b}{a}=A\frac{B}{A^2}$, yields
$$
\frac{\theta(z-P)}{\theta_U(z)}\left(\frac{\theta_{UU}(z)}{\theta(z-P)}- 2\frac{\theta_U(z-P) \theta_U(z)}{ \theta^2(z-P)}\right)
$$
$$=
-\frac{\theta(z+P)}{\theta_U(z)}\left(\frac{\theta_{UU}(z)}{\theta(z+P)}- 2\frac{\theta_U(z+P) \theta_U(z)}{ \theta^2(z+P)}\right)
$$
which upon clearing the denominators gives the required identity on the theta divisor $\lbrace\theta(z)=0\rbrace$.
\end{proof}

\begin{proof}[Proof of the main theorem]
We note that up to replacing $U$ by $V$ equation (\ref{thetaidentity}) is identical to equation (1.7) in \cite{krtri}, which is shown there to be equivalent to the semidegenerate case of the trisecant conjecture. In \cite{krtri} it is then shown that the condition of having a semidegenerate trisecant characterizes Jacobians, and thus (\ref{thetaidentity}) is equivalent to the ppav being a Jacobian (notice that we also obtain anindirect proof of proposition \ref{jacOk} in this way.
\end{proof}

\begin{rem}
The original statement of the $\Gamma_{00}$ conjecture can also be reformulated in terms of the theta function. Indeed, if the matrix of coefficients of linear dependency is of rank $N$, i.e. if we have
$$
 c_{ij}=\sum_{n=1}^N U^{(n)}_i V^{(n)}_i+V^{(n)}_i U^{(n)}_i,
$$
then the functions $\tau$ and $\psi$ would be obtained similarly, with
$$
  \tau(x_1,y_1,\ldots,x_n,y_n,t):=\theta\left(\sum\limits_{n=1}^N U^{(n)}x_n+V^{(n)}y_n+Pt+Z\right),
$$
instead of (\ref{taudef}) and $\psi$ still given by (\ref{psidef}), while $u$ will take the form
$$
  u(x,y,t):=2\sum\limits_{n=1}^N\p_{x_n}\p_{y_n}\ln\tau(x_1,y_1,\ldots,x_n,y_n,t),
$$
instead of (\ref{udef}), and equation (\ref{upsi}) from the main theorem will then still be satisfied. However, in this case the Taylor expansion for $u$ in terms of $x_1$ near $x_1=\eta(y_1,\ldots,x_n,y_n,t)$ would contain a term
$$
 \sum_{n=2}^N\frac{\eta_{x_ny_n}}{x_1-\eta},
$$
of order $-1$, which would result in the computations to completely collapse. At the moment we do not see any way to extend our approach to prove the original $\Gamma_{00}$. Note that when thinking of $\Gamma_{00}$ as a limiting case of semidegenerate trisecants, following proposition \ref{jacOk}, we only get our special case, and thus potentially there could be abelian varieties that are not Jacobians violating the original $\Gamma_{00}$ conjecture (though of course at the moment we are not aware of any such examples or candidates).

Note that for $g=4$ the matrix of coefficients $c_{ij}$ of the linear dependence for the two extra points $\pm (a-a')$, where $a,a'$ are the $g_3^1$'s on the curve (see \cite{W2}), is of rank greater than two.
\end{rem}

\section*{Acknowledgements}
We are grateful to Igor Krichever for discussions on using integrable systems techniques for characterizing certain loci of abelian varieties, and to Gerard van der Geer for comments on the $\Gamma_{00}$ conjecture and on the first version of this text. We would like to thank Jos\'e Mar\'\i a Mu\~noz Porras for pointing out that in theorem 1 the matrix of coefficients, being symmetric, can be allowed to be of rank two and not only one.

\end{document}